\documentclass[11pt,reqno,a4paper]{amsart}
\usepackage{hyperref}
\usepackage{amsmath}
\usepackage{amssymb,a4wide}
\usepackage{amsthm}
\usepackage[margin=3cm]{geometry}
\usepackage{graphicx,epstopdf}
\usepackage{enumerate}

\hypersetup{
    colorlinks=true,
    linkcolor=black,
    filecolor=black,      
    urlcolor=blue,
    citecolor=red,
}

\newcommand{\guio}[1]{\nobreakdash-\hspace{0pt}}

\numberwithin{equation}{section}
\newtheorem{theorem}{Theorem}[section]
\newtheorem{lemma}[theorem]{Lemma}
\newtheorem{prop}[theorem]{Proposition}

\theoremstyle{definition}
\newtheorem{example}[theorem]{Example}

\newtheorem{remark}[theorem]{Remark}

\newcommand{\R}{\mathbb{R}}

\renewcommand{\S}{\mathbb{S}}

\newcommand{\supp}{\operatorname{supp}}
\newcommand{\HH}{\mathcal{H}^2}
\newcommand{\mres}{\mathbin{\vrule height 1.6ex depth 0pt width
0.13ex\vrule height 0.13ex depth 0pt width 1.3ex}}

\title[Explicit minimisers for anisotropic Coulomb energies in 3D]
{Explicit minimisers\\ for anisotropic Coulomb energies in 3D
}
\author[J.~Mateu]{J.~Mateu}
\author[M.G.~Mora]{M.G.~Mora}
\author[L.~Rondi]{L.~Rondi}
\author[L.~Scardia]{L.~Scardia}
\author[J.~Verdera]{J.~Verdera}

\address[J. Mateu]{Department de Matem\`atiques, Universitat Aut\`onoma de Barcelona, and Centre de Recerca Matem\`atica, Catalonia}
\email{mateu@mat.uab.cat}

\address[M.G. Mora]{Dipartimento di Matematica, Universit\`a di Pavia, Italy}
\email{mariagiovanna.mora@unipv.it}

\address[L. Rondi]{Dipartimento di Matematica, Universit\`a di Pavia, Italy}
\email{luca.rondi@unipv.it}

\address[L. Scardia]{Department of Mathematics, Heriot-Watt University, United Kingdom}
\email{L.Scardia@hw.ac.uk}

\address[J. Verdera]{Department de Matem\`atiques, Universitat Aut\`onoma de Barcelona, and Centre de Recerca Matem\`atica, Catalonia}
\email{jvm@mat.uab.cat}

\date{}

\begin{document}

\begin{abstract}
In this paper we consider a general class of anisotropic energies in three dimensions and give a complete characterisation of their minimisers. We show that, depending on the Fourier transform of the interaction potential, the minimiser is either the normalised characteristic function of an ellipsoid or a measure supported on a two-dimensional ellipse. In particular, it is always an ellipsoid if the transform is strictly positive, while when the Fourier transform is degenerate both cases can occur. Finally, we show an explicit example where loss of dimensionality of the minimiser does occur.

\bigskip


\medskip

\end{abstract}

\maketitle

\section{Introduction}
We consider the nonlocal interaction energy 
\begin{equation}\label{en:general}
I(\mu):= \int_{\R^3} (W\ast\mu)(x) \,d\mu(x) +\int_{\R^3}|x|^2 \,d\mu(x),
\end{equation}
defined on probability measures $\mu\in\mathcal{P}(\R^3)$, where the potential $W$ is of the form 
\begin{equation}\label{potdef}
W(x):=\frac1{|x|}\Psi\left(\frac{x}{|x|}\right),\quad x\in \R^3,\ x\neq 0,
\end{equation}
 with the profile $\Psi$
even, strictly positive, and smooth enough. Note that $W$ is an anisotro\-pic extension of the classical Coulomb potential, which corresponds to a constant profile~$\Psi$.

Our {\em main result} states that if the Fourier transform of $W$ is nonnegative, then the corresponding energy $I$ has a unique minimiser which is either 
\begin{equation}\label{ellipsoidlaw-intro}
\mu_E(x):=\frac{\chi_E(x)}{|E|}, 
\end{equation}
for an ellipsoid $E\subset \R^3$, or, up to a rotation,
\begin{equation}\label{s-ellipsoidlaw-intro}
\mu_{\tilde{E}}(x):=\frac{3}{2\pi a_1a_2}\sqrt{1-\frac{x_1^2}{a_1^2}-\frac{x_2^2}{a_2^2}}\,  \mathcal{H}^2\mres \tilde{E}\otimes \delta_0(x_3),
\end{equation}
for an ellipse $\tilde{E}\subset \R^3$  with semi-axes of length $a_1$ and $a_2$ on the coordinate axes of the $x_1x_2$-plane. 
We refer to the measure in \eqref{ellipsoidlaw-intro} as an {\em ellipsoid law} and to the lower-dimensional measure in \eqref{s-ellipsoidlaw-intro} as a {\em semi-ellipsoid law}. 

More precisely, we show in Theorem \ref{Mainthm} that if the Fourier transform of $W$ is strictly positive, then the minimiser is an ellipsoid law, hence a measure with a three-dimensional support. If instead the Fourier transform of $W$ is not strictly positive, the minimiser is either an ellipsoid law, or a semi-ellipsoid law, and both cases can occur, as we show in Example ~\ref{ex:loss-dim}.

We recall that a similar phenomenon occurs in two dimensions. In that case, if the Fourier transform of the potential is not strictly positive, the minimiser is   either the normalised characteristic function of an ellipse, which we call an ellipse law, or, up to a rotation, the semi-circle law measure 
\begin{equation*}
\mu_{\textrm{sc}}(x)=\frac1\pi\sqrt{2-x_1^2} \, \mathcal{H}^1\mres(-\sqrt 2, \sqrt 2)\otimes\delta_0(x_2),
\end{equation*}
and both cases are possible. Note that in the three-dimensional case degeneracy of the support can happen only on a two-dimensional set, and not on a set of dimension lower than two.

Our main result Theorem~\ref{Mainthm} is the most general result to date for {\em three-dimensional anisotropic} energies, and provides a full characterisation of their  minimisers in both the non-degenerate and degenerate cases, under minimal assumptions on $W$.

The motivation of our paper is twofold.  There is a large recent literature on nonlocal energies, due to the pivotal role they play in describing the behaviour of large systems of interaction particles, in a variety of applications. Traditionally, the focus of the mathematical literature on nonlocal energies has been on {\em radially symmetric} potentials, which model interactions depending on the mutual distance between the particles (see e.g. \cite{BCLR13, CCP15, CHV19, CSW-preprint, FYK11, SST}). The mathematical study of anisotropic potentials, despite their natural occurrence in modelling interactions where a preferred direction of interaction is present, has on the other hand been very limited until recently. The potential \eqref{potdef} is the natural anisotropic extension of the Coulomb kernel. Our main motivation is the characterisation of the minimisers of the corresponding anisotropic energies. In particular, the issue of determining the {\em dimension} of the support of minimisers has so far remained elusive.  

Also, ellipses and more generally {\em ellipsoids} play a special role in several problems: as minimisers of Coulomb and dislocation energies \cite{Fro, MRS, CMMRSV, CMMRSV2}, as special solutions (vortex patches) of the Euler equations \cite{Joans}, as the only bounded coincidence sets (with non-empty interior) of global solutions of the obstacle problem \cite{Dive31}, as special sets in Calder\'on-Zygmund theory \cite{MOV}. It is then natural to try and characterise the potentials for which minimisers of energies like \eqref{en:general} are ellipsoids.

In the recent paper \cite{CS2d}, Carrillo and Shu gave a complete answer to the minimality of ellipses in two dimensions, and partially solved the dimensionality conundrum. It is indeed proved in \cite{CS2d} that, under the two-dimensional analogue of the assumptions we make in Theorem~\ref{Mainthm}, if the potential has a strictly positive Fourier transform, then the minimiser is a non-degenerate ellipse. If instead the Fourier transform vanishes in some direction, then the minimiser is either an ellipse law or a semi-circle law measure, and both cases can occur. In this paper we prove the analogous result in the three-dimensional case by means of a different approach which, spelled out in the two-dimensional case, provides an alternative proof of the result in \cite{CS2d}. We believe that our new approach is robust enough to be generalised to higher dimension. This is not a line of investigation that we pursue here, but will be addressed in a forthcoming paper.

Compared with the most recent results for three-dimensional anisotropic energies, \cite{MMRSV-Pert} and \cite{CS3d}, the present paper represents a substantial improvement in several directions. In our previous work \cite{MMRSV-Pert} we proved the minimality of ellipsoids under the additional  requirement that the potential $W$ is a small perturbation of the Coulomb potential $W_{\textrm{C}}$. Since the perturbation assumption guarantees the strict positivity of the Fourier transform of $W$, the analysis there was limited to a special case of our result in the non-degenerate case. We now {\em lift the perturbation assumption}, and also allow for potentials whose Fourier transform is only nonnegative, hence facing the issue of loss of dimensionality for the minimiser. 
As for \cite{CS3d}, the approach there was heavily based on the two-dimensional case \cite{CS2d}. Indeed, the authors impose strong symmetry assumptions on the potentials, which reduce the complexity of the problem to a two-dimensional setting. 
By following a different approach, in Theorem \ref{Mainthm} we {\em remove the additional symmetry assumptions} in \cite{CS3d}, and provide a full characterisation of the minimisers in the general setting. Moreover, we also show that the minimiser cannot concentrate on a one-dimensional set, a case that was not excluded in~\cite{CS3d}.

\subsection{Literature review} The recent rise of interest in {\em anisotropic} energies has been propelled by the works \cite{MRS,CMMRSV} (see also \cite{MMRSV20,MMRSVIzv}), where the two-dimensional energy 
\begin{equation*}
I_\alpha(\mu):=\int_{\R^2}(W_\alpha \ast \mu)(x) d\mu(x) +\int_{\R^2}|x|^2 \,d\mu(x),
\end{equation*}
was considered. Here the interaction potential is $W_\alpha(x)= W_{\textrm{C}}(x) + \alpha(x_1/|x|)^2$, $\alpha\in [-1,1]$. 

The case $\alpha=1$ is motivated by the study of edge dislocations in metals, lattice defects whose motion is restricted to a direction $\mathbf{b}\in \mathbb{S}^1$ in the lattice. The potential $W_1$ indeed describes the interaction among dislocations in the case $\mathbf{b}=e_1$, and the corresponding energy $I_1$ models their interaction at a continuum scale, where $\mu$ represents the macroscopic density of such defects. 

The energy $I_\alpha$ admits a unique minimiser, which has been fully characterised in \cite{MRS,CMMRSV}: For $\alpha\in (-1,1)$ it is the normalised characteristic function of an ellipse, and it degenerates from an ellipse to the semi-circle law measure on the vertical axis for $\alpha=1$, and on the horizontal axis for $\alpha=-1$. 

The study of the energy $I_\alpha$ prompted several compelling questions, and motivated the analysis in \cite{CMMRSV2,MMRSV-Pert, MMRSV-Pert2,CS2d, CS3d} and in the current paper.  A first question is whether the previous analysis is bound to the special structure of $W_\alpha$ or can be extended to more general anisotropies. In \cite{MMRSV-Pert} we gave a partial answer to this question in arbitrary dimension. We proved that ellipsoids arise as energy minimisers whenever  $W$ is a small, even perturbation of the Coulomb potential $W_{\textrm{C}}$ with the same homogeneity of $W_{\textrm{C}}$. In three dimensions this corresponds to considering $W$ of the form \eqref{potdef}, with $\Psi$ close to $1$. This additional assumption was instrumental to the proof we provided, where we resorted to the use of the Implicit Function Theorem, but we \textit{conjectured} in \cite{MMRSV-Pert2} that the result was likely to hold true without it.

This was indeed demonstrated, in the \textit{two-dimensional} case, in the recent breakthrough paper \cite{CS2d} by Carrillo and Shu. 
The approach followed in \cite{CS2d} is based on an ingenious way of rewriting the energy in terms of a one-dimensional profile, and then deduce the minimality of ellipses from the minimality of the semi-circle law for the one-dimensional logarithmic energy proved in \cite{Wigner}. 

Another question of interest is what causes a loss of dimensionality for minimisers. The result in \cite{CS2d} and our paper show that, in two and three dimensions, a necessary condition for this to happen is the degeneracy of the Fourier transform. However there is currently no criterion indicating which shape will be preferred by the minimiser in the degenerate case. Indeed, there are explicit examples of anisotropic energies with degenerate Fourier transform for which the minimiser has a full-dimensional support (see \cite{Giorgio} in the two-dimensional case and \cite{CMMRSV2} in the three-dimensional case). An instance of loss of dimensionality in three dimensions is shown in Example~\ref{ex:loss-dim}.

\subsection{Idea of the proof of the main result} We sketch here the idea of the proof of the main result, Theorem \ref{Mainthm}. The existence of  a compactly supported minimiser follows by standard methods. Moreover, the sign condition on the Fourier transform of $W$ implies that the energy is strictly convex, and hence there is a unique minimiser of $I$. We can then focus on the characterisation of such minimiser, which is equivalent to finding a solution of the Euler-Lagrange conditions, due to the strict convexity of the energy. 

 Motivated by the perturbation result \cite{MMRSV-Pert} and by \cite{CS3d}, we look for a candidate ellipsoid $E\subset \R^3$ such that its normalised characteristic function $\chi:=\chi_E/|E|$ satisfies the Euler-Lagrange conditions
\begin{align}\label{EL1-intro}
\left(W\ast \chi \right) (x)+ \frac{1}{2}|x|^2&= C \quad \text{for }  x \in E,\\
\left(W\ast \chi  \right) (x)+ \frac{1}{2}|x|^2&\ge C \quad \text{for }x \in \R^3 \setminus E. \label{EL2-intro}
\end{align}
A crucial step in our approach is to resort to a Fourier representation of the potential $W\ast \chi$, appearing at the left-hand side of \eqref{EL1-intro}--\eqref{EL2-intro} (see Section~\ref{sub:F}). One of the advantages of this representation is that it shows directly that $W\ast \chi$ is quadratic on $E$, which is a necessary condition for \eqref{EL1-intro} to hold. Moreover, in Fourier terms it is immediate to see that if an ellipsoid satisfies the stationarity condition \eqref{EL1-intro}, then it satisfies also \eqref{EL2-intro}. Hence it remains to find a stationary ellipsoid. 

In Fourier terms, the stationarity condition \eqref{EL1-intro} can be equivalently expressed as the system
\begin{equation}\label{EL1-Fourier-intro-d}
\frac{3}{4\sqrt{2\pi}}
\int_{\mathbb{S}^2}\frac{\omega_i\omega_j\widehat{W}(\omega)}{(M\omega\cdot\omega)^{3/2}}\,d\mathcal{H}^2(\omega)=\delta_{ij},
\quad \text{ for } i,j=1,2,3,
\end{equation}
where $M$ is a positive-definite symmetric matrix which incorporates the information about the semi-axes of the ellipsoid and its orientation. A key observation is that a solution 
to the system \eqref{EL1-Fourier-intro-d} is a critical point of the auxiliary scalar function $f$ which is 
defined in \eqref{def:function-f} for positive symmetric matrices $M$. To find a solution of \eqref{EL1-Fourier-intro-d} we then study the auxiliary function $f$ and show that, when $\widehat W>0$ on $\mathbb{S}^2$, it admits a minimiser $M$. 

When instead   $\widehat W\geq0$, we repeat the procedure above for a perturbed potential with strictly positive Fourier transform, and by letting the perturbation parameter to zero, we investigate the possible stationary shapes in the limit. Energy considerations allow us to exclude that the limiting measure concentrates on points or on segments, hence only the ellipsoid and semi-ellipsoid law measures are attainable. We then show that the semi-ellipsoid law can be attained, by providing an explicit example.

\subsection{Plan of the paper} The paper is organised as follows. In Section~\ref{sect:prelim} we collect some preliminary results that will be used in the paper. In particular, in Section~\ref{sub:F} we compute the Fourier representation of $W\ast \chi$, which is a crucial step of our approach. Section~\ref{sec:main-res} is the heart of the paper and contains the proof of the main result, Theorem \ref{Mainthm}.

\section{Preliminaries}\label{sect:prelim}
In this section we collect some definitions and preliminary results that will be needed in the paper. We also establish some notation.

\subsection{Spherical harmonics and Fourier transforms}
On $L^2(\mathbb{S}^{2})$ we consider the Fourier basis given by the spherical harmonics, which are obtained as the restriction to $\mathbb{S}^2$ of homogeneous harmonic polynomials of degree $k\geq 0$.
For $k\geq 0$, the dimension of the space of spherical harmonics of degree $k$ in $\R^3$ is $2k+1$.
In particular, any
$\varphi\in L^2(\mathbb{S}^{2})$ can be decomposed as
\begin{equation}\label{Fser}
\varphi=\sum_{k=0}^{\infty} a_k\Phi_k,
\end{equation}
where $\Phi_k$ is a spherical harmonic of degree $k\geq 0$ with
\begin{equation}\label{unitary-L2}
\|\Phi_k\|_{L^2(\mathbb{S}^{2})}=1
\end{equation}
and $\{a_k\}_{k\geq 0}\in \ell^2$. We note that $\|\varphi\|_{L^2(\mathbb{S}^{2})}=\|\{a_k\}_{k\geq 0} \|_{\ell^2}$.

We recall that for any homogeneous harmonic polynomial $\Phi_k$ of degree $k\geq 0$ we have the inequality
\begin{equation}\label{infty2spher}
\|\Phi_k\|_{L^{\infty}(\mathbb{S}^{2})}\leq \sqrt{\frac{2k+1}{4\pi}}\|\Phi_k\|_{L^2(\mathbb{S}^{2})},
\end{equation}
see \cite[Proposition~4.16]{Eft-Frye}.

We now recall the definition of Sobolev spaces on the sphere, see for instance \cite{Bar-et-al}. Let $\Delta_S$ denote the Laplace-Beltrami operator on $\mathbb{S}^{2}$.
For $s>0$, we denote by $H^s(\mathbb{S}^{2})$ the subspace of all $\varphi\in L^2(\mathbb{S}^{2})$ such that
$(-\Delta_S)^{s/2}\varphi\in L^2(\mathbb{S}^{2})$.
We identify $H^0(\mathbb{S}^{2})$ with $L^2(\mathbb{S}^{2})$.
On $H^s(\mathbb{S}^{2})$, $s>0$, 
we consider the norm
$$\|\varphi\|_{H^s(\mathbb{S}^{2})}^2=\|\varphi\|_{L^2(\mathbb{S}^{2})}^2+\|(-\Delta_S)^{s/2}\varphi\|_{L^2(\mathbb{S}^{2})}^2.$$
For any $s\geq 0$, the space $H^s(\mathbb{S}^{2})$ can be equivalently defined as the subspace of all $\varphi\in L^2(\mathbb{S}^{2})$ such that
$$\big\{\big(1+\sqrt{k(k+1)}\big)^{s}a_k\big\}_{k\geq 0}\in \ell^2$$
(where $\{a_k\}_{k\geq 0}$ is as in \eqref{Fser}) with equivalent norm
$$
\|\varphi\|_{H^s(\mathbb{S}^{2})}=\big\|\big\{\big(1+\sqrt{k(k+1)}\big)^{s}a_k\big\}_{k\geq 0} \big\|_{\ell^2}.
$$

Moreover, we have the following embedding.

\begin{lemma}\label{immersion}
The space $H^s(\mathbb{S}^{2})$ is continuously embedded in $C^0(\mathbb{S}^{2})$ for any $s>1$.
\end{lemma}

\begin{proof}
It is enough to prove that
\begin{equation}\label{emb_C}
\|\varphi\|_{L^{\infty}(\mathbb{S}^{2})}\leq C(s)\|\varphi\|_{H^s(\mathbb{S}^{2})}
\end{equation}
for every $\varphi\in H^s(\mathbb{S}^{2})$ and any $s>1$, where $C(s)>0$ is a constant depending on $s$ only. 
Indeed, if \eqref{emb_C} holds true, then it is easy to infer that $\varphi\in H^s(\mathbb{S}^{2})$ is continuous, since it is the uniform limit of continuous functions, and the proof is then complete.

By \eqref{Fser}, \eqref{unitary-L2}, and \eqref{infty2spher} we have that
\begin{eqnarray*}
\|\varphi\|_{L^{\infty}(\mathbb{S}^{2})} & \leq & \sum_{k=0}^{\infty}|a_k|\|\Phi_k\|_{L^{\infty}(\mathbb{S}^{2})}
\ \leq \
\frac{1}{\sqrt{4\pi}}\sum_{k=0}^{\infty}\sqrt{2k+1}|a_k| \\
& = &
\frac{1}{\sqrt{4\pi}}\sum_{k=0}^{\infty}\frac{\sqrt{2k+1}}{\big(1+\sqrt{k(k+1)}\big)^{s}}\big(1+\sqrt{k(k+1)}\big)^{s}|a_k|\\
& \leq & \frac{1}{\sqrt{4\pi}}\Bigg(\sum_{k=0}^{\infty}\frac{2k+1}{\big(1+\sqrt{k(k+1)}\big)^{2s}}\Bigg)^{1/2}
\|\varphi\|_{H^s(\mathbb{S}^{2})}.
\end{eqnarray*}
Since the series above is convergent for $s>1$, the proof is complete.
\end{proof}

The Fourier transform definition we adopt is 
$$
\widehat{f}(\xi) =\frac{1}{(2\pi)^{3/2}}\int_{\R^3} f(x) e^{-i \xi \cdot x}\, dx, \quad \xi \in \R^3,
$$
for functions $f$ in the Schwartz space $\mathcal{S}$.
Correspondingly, the inverse Fourier transform is the following
$$
f(x)=\frac{1}{(2\pi)^{3/2}} \int_{\R^3} \widehat{f}(\xi) e^{i \xi \cdot x}\, d\xi, \quad x \in \R^3.
$$

Finally, we recall the expression of the Fourier transform for homogeneous harmonic polynomials. If $\Phi_k$ is a homogeneous harmonic polynomial of degree $k\geq 1$ and $0< s < 3$, then
\begin{equation*}
\frac{\Phi_k(x)}{|x|^{3-s+k}}  \xrightarrow{\rm{Fourier}} (-i)^{k} 2^{s-3/2} \frac{\Gamma(\frac{k+s}{2})}{\Gamma(\frac{k+3-s}{2})}\,\frac{\Phi_k(\xi)}{|\xi|^{k+s}},
\end{equation*}
see \cite[Chapter~III, Theorem~5]{S} where a slightly different definition of the Fourier transform is adopted.
In particular, for $s=2$, we have
\begin{equation}\label{foupolk2}
\frac{\Phi_k(x)}{|x|^{1+k}}  \xrightarrow{\rm{Fourier}} (-i)^{k} \sqrt{2} \frac{\Gamma(\frac{k}{2}+1)}{\Gamma(\frac{k+1}{2})}\,\frac{\Phi_k(\xi)}{|\xi|^{k+2}}
=:b_k \,\frac{\Phi_k(\xi)}{|\xi|^{k+2}}.
\end{equation}
The formula \eqref{foupolk2} holds true also for $k=0$, and can be proved by a classical argument based on the fundamental solution of the Laplacian operator.


\subsection{The Fourier transform of $W$.}
Let $W$ be as in \eqref{potdef} with $\Psi$ even and in $H^s(\mathbb{S}^2)$ for $s>\frac32$. 
Since $\Psi$ is continuous, the kernel $W$ is locally $L^1$ and bounded at infinity, therefore it is a tempered distribution. We now write its Fourier transform. 

To this aim, we first write $W$ in terms of spherical harmonics, using \eqref{potdef} and the decomposition \eqref{Fser} for $\Psi$. 
Since $\Psi$ is even, the decomposition \eqref{Fser} for $\Psi$ reduces to
\begin{equation}\label{phieven}
\Psi=\sum_{k=0}^{\infty} a_{2k}\Phi_{2k}.
\end{equation}
Hence, for $x\neq 0$, we have 
$$
W(x) = \sum_{k=0}^{\infty} a_{2k}\frac{\Phi_{2k}(x)}{|x|^{1+2k}}.
$$
By \eqref{foupolk2} we have 
\begin{equation}\label{foupolk3}
\frac{\Phi_{2k}(x)}{|x|^{1+2k}}  \xrightarrow{\rm{Fourier}} (-1)^{k}\sqrt{2} \frac{\Gamma(k+1)}{\Gamma(k+\frac{1}{2})}\,\frac{\Phi_{2k}(\xi)}{|\xi|^{2k+2}}
= b_{2k}\frac{\Phi_{2k}(\xi)}{|\xi|^{2k+2}}.
\end{equation}
Note that, since $$\Gamma\left(k+\tfrac12\right)=\frac{\sqrt{\pi}(2k)!}{4^k k!}\quad\text{and}\quad \Gamma(k+1)=k!,$$
we may rewrite
\begin{equation}\label{bcoeff}
b_{2k}=\sqrt{\frac2{\pi}}(-1)^k4^k\frac{k!k!}{(2k)!}.
\end{equation}
Moreover,
by Stirling's formula,
$|b_{2k}|\sim  \sqrt{2k}$, hence, for an absolute constant $C$,
$$
|b_{2k}| \leq C(1+\sqrt{2k(2k+1)})^{1/2},\quad k\geq 0.
$$
Since $\Psi\in H^{s}(\mathbb{S}^{2})$ for $s>\frac32$, we deduce that
\begin{equation}\label{hatW}
\widehat{W}(\xi)=\frac1{|\xi|^2}\widehat{\Psi}(\xi/|\xi|)
\end{equation}
where, with a little abuse of notation,
\begin{equation}\label{phievenbis}
\widehat{\Psi}:=\sum_{k=0}^{\infty} a_{2k}b_{2k}\Phi_{2k}\in H^{s-\frac12}(\mathbb{S}^{2}).
\end{equation}
Note that since $s-\frac12>1$, the function $\widehat\Psi$ is continuous by Lemma~\ref{immersion}.

\subsection{Ellipses and ellipsoids} For any $a=(a_1,a_2,a_3)\in \R^3$, $D(a)=\mathrm{diag}(a_1,a_2,a_3)$ stands for the diagonal matrix such that $D_{ii}=a_i$ for $i=1,2,3$. Given $a=(a_1,a_2,a_3)\in \R^3$ with $a_i>0$ for $i=1,2,3$, we let 
\begin{equation}\label{el}
E_0(a):= \left\{x=(x_1,x_2,x_3)\in \R^3:\ \frac{x_1^2}{a_1^2}+\frac{x_2^2}{a_2^2}+\frac{x_3^2}{a_3^2} \le 1\right\}
\end{equation}
denote the compact set enclosed by the ellipsoid with semi-axes of length $a_1$, $a_2$, and $a_3$ on the coordinate axes. Note that 
\begin{equation*}
E_0(a) = D(a)\overline{B},
\end{equation*}
where $B$ denotes the unit ball $B_1(0)\subset\R^3$.
A general ellipsoid $E\subset \R^3$ centred at the origin can be then obtained by rotating $E_0(a)$ in \eqref{el} with respect to the coordinate axes, namely as
\begin{equation}\label{ellrot}
E= R E_0(a)=R D(a) \overline{B},
\end{equation}
for some rotation $R\in SO(3)$.

Similarly, for given constants $a_1>0$ and $a_2>0$, we let
\begin{equation}\label{ellipse}
\tilde{E}_0(a_1,a_2):=\left\{(x_1,x_2)\in \R^2:\ \frac{x_1^2}{a_1^2}+\frac{x_2^2}{a_2^2}  \leq 1\right\}\times \{0\}
\end{equation}
denote 
the compact set enclosed by the ellipse with semi-axes of length $a_1$ and $a_2$ on the coordinate axes of the $x_1x_2$-plane. 
A general two-dimensional ellipse $\tilde{E}\subset \R^3$ centred at the origin can be then obtained by rotating $\tilde{E}_0(a_1,a_2)$ in \eqref{ellipse} with respect to the coordinate axes, namely as
\begin{equation}\label{ellipserot}
\tilde{E}= R \tilde{E}_0(a_1,a_2),
\end{equation}
for some rotation $R\in SO(3)$.

\subsection{The Fourier transform of the characteristic function of an ellipsoid}
The Fourier transform of the characteristic function of $B$ in $\R^3$ is given by
$$\widehat{\chi_{B}}(\xi)=\frac{J_{3/2}(|\xi|)}{|\xi|^{3/2}},$$
where $J_{\alpha}$ denotes the Bessel function of the first kind of order $\alpha$.
We recall that
$$J_{3/2}(r)=\sqrt{\frac{2}{\pi r}}\Big(\frac{\sin r}r-\cos r\Big),\quad r>0,$$
hence
\begin{equation}\label{chiB-J}
\widehat{\chi_{B}}(\xi)=
\sqrt{\frac{2}{\pi }}\Big(\frac{\sin |\xi|}{|\xi|^3}-\frac{\cos |\xi|}{|\xi|^2}\Big).
\end{equation}
Moreover, we have the following asymptotic behaviours:
$$\frac{J_{3/2}(r)}{r^{3/2}}\sim \frac1{2^{3/2}\Gamma(5/2)}=
\frac13\sqrt{\frac{2}{\pi}},
\qquad \text{as }r\to 0^+,$$
and
$$\frac{J_{3/2}(r)}{r^{3/2}}\sim -\sqrt{\frac{2}{\pi }}\frac1{r^{2}}\cos r,\qquad \text{as }r\to +\infty,$$
see for instance \cite[Section~5.16]{Lebedev}.

Let $E$ be an ellipsoid of the form \eqref{ellrot}. If we set $\chi:=\chi_E/|E|$, then it is immediate to see that
\begin{equation}\label{chiE-chiB}
\widehat\chi(\xi)=\frac1{|B|}\widehat{\chi_{B}}(D(a)R^T\xi).
\end{equation}

\subsection{The Fourier representation of $W\ast \chi$}\label{sub:F}
In this section, we assume that $\Psi$ is even and in $H^s(\mathbb{S}^2)$, for $s>\frac32$. We note that here we {\em do not require} any sign condition on $\Psi$.

Let $E$ be an ellipsoid of the form \eqref{ellrot}, and let $\chi:=\chi_E/|E|$. We express $W\ast \chi$ in Fourier terms.
By \cite[Chapter~VI.3, Theorem~6]{Yosida} and by using \eqref{chiB-J} and \eqref{chiE-chiB}, we have
\begin{eqnarray*}
\widehat{W\ast\chi}(\xi) & = & (2\pi)^{3/2}\widehat{W}(\xi) \widehat{\chi}(\xi)
\\
& = & \frac{4\pi}{|B|}
\Bigg(\frac{\sin (|D(a)R^T\xi|)}{|D(a)R^T\xi|^3}-\frac{\cos( |D(a)R^T\xi|)}{|D(a)R^T\xi|^2}\Bigg)\frac1{|\xi|^2}\widehat{\Psi}(\xi/|\xi|)\in L^1(\mathbb{R}^3),
\end{eqnarray*}
with $\Psi$ as in \eqref{phieven} and $\widehat{\Psi}$ as in \eqref{phievenbis}.
Hence, the inversion formula holds, that is,
$$
(W\ast\chi)(x)=\int_{\R^3}\widehat W(\xi)\widehat\chi(\xi)e^{ix\cdot\xi}\, d\xi=\int_{\R^3}\widehat W(\xi)\widehat\chi(\xi)\cos(x\cdot\xi)\, d\xi
$$
for every $x\in\R^3$. Writing this integral in spherical coordinates, we obtain
\begin{align*}
&(W\ast\chi)(x) \\
& =\frac{1}{|B|}\sqrt{\frac2{\pi}}
\int_{\mathbb{S}^2}\int_0^{\infty}\Big(\frac{\sin (r|D(a)R^T\omega|)}{r^3|D(a)R^T\omega|^3}-\frac{\cos( r|D(a)R^T\omega|)}{r^2|D(a)R^T\omega|^2}\Big)\widehat{\Psi}(\omega)\cos(x\cdot r\omega)\,dr\,d\mathcal{H}^2(\omega).
\end{align*}
Setting
\begin{equation}\label{alpha-rho}
\rho:=r |D(a)R^T\omega|\quad\text{and}\quad \alpha(x,\omega):=\frac{x\cdot \omega}{|D(a)R^T\omega|},
\end{equation} 
we have
$$
(W\ast\chi)(x)=\frac{1}{|B|}\sqrt{\frac2{\pi}}
\int_{\mathbb{S}^2}\Bigg(\int_0^{\infty}\Big(\frac{\sin \rho}{\rho^3}-\frac{\cos\rho}{\rho^2}\Big)\cos( \alpha(x,\omega)\rho)\,d\rho\Bigg)\frac{\widehat{\Psi}(\omega)}{{|D(a)R^T\omega|}}\,d\mathcal{H}^2(\omega).
$$
By several integration by parts we obtain that for any $\alpha\in\R$
$$\int_0^{\infty}\Big(\frac{\sin \rho}{\rho^3}-\frac{\cos\rho}{\rho^2}\Big)\cos(\alpha\rho )\,d\rho=\frac12(1-\alpha^2)\int_0^{\infty}\frac{\sin \rho\cos(\alpha\rho)}{\rho} \,d\rho.
$$
Computing this improper integral is a simple exercise. Alternatively, one can appeal to formula 1.6(1) on page 18 of \cite{Bateman} and obtain 
$$\int_0^{\infty}\Big(\frac{\sin \rho}{\rho^3}-\frac{\cos\rho}{\rho^2}\Big)\cos(\alpha\rho )\,d\rho=
\frac{\pi}4(1-\alpha^2)\chi_{(-1,1)}(\alpha).
$$
Finally, we have that
\begin{equation}\label{pot}
(W\ast\chi)(x) = 
\frac{1}{|B|}\sqrt{\frac{\pi}8}
\int_{\mathbb{S}^2}(1-\alpha^2(x,\omega))\chi_{(-1,1)}(\alpha(x,\omega))\frac{\widehat{\Psi}(\omega)}{{|D(a)R^T\omega|}}\,d\mathcal{H}^2(\omega).
\end{equation}

The above formula shows that $W\ast\chi$ is a quadratic polynomial in $E$, since 
for $x$ in the interior of $E$ one has that
$|\alpha(x,\omega)|<1$ for any $\omega\in\mathbb{S}^2$. This can be seen by rewriting 
$$\alpha(x,\omega)=(D(a)^{-1}R^Tx)\cdot (D(a)R^T\omega/|D(a)R^T\omega|)=:y\cdot \tilde\omega,$$
where $\tilde\omega\in \mathbb{S}^2$. It is then immediate to see that $x\in E$ if and only if $y\in \overline{B}$. 
Hence, for $x$ in the interior of $E$ we have that
$|\alpha(x,\omega)|<1$ for any $\omega\in\mathbb{S}^2$. We also note that for $x\in\partial E$ we have that $|\alpha(x,\omega)|<1$ except on a set of $\omega\in\mathbb{S}^2$ of $\mathcal{H}^{2}$-measure zero, whereas for $x\notin E$ we have that $|\alpha(x,\omega)|\neq 1$ except on a set of $\mathcal{H}^{2}$-measure zero.

We conclude this section by computing the gradient and the Hessian of $W\ast\chi$, which will be used in the proof of Theorem~\ref{Mainthm}.
For the gradient, since the function $t\in\R\mapsto (1-t^2)\chi_{(-1,1)}(t)$ is a Lipschitz function, we can differentiate \eqref{pot}
with respect to $x$. This provides us with the following formula: for any $j=1,2,3$ and any $x\in \R^3$
\begin{equation}\label{potgradbis-0}
\frac{\partial(W\ast\chi)}{\partial x_j}(x)
=-\frac{1}{|B|}\sqrt{\frac{\pi}2}
\int_{\mathbb{S}^2}\frac{\alpha(x,\omega)\,\omega_j}{|D(a)R^T\omega|}\chi_{(-1,1)}(\alpha(x,\omega))\frac{\widehat{\Psi}(\omega)}{{|D(a)R^T\omega|}}\,d\mathcal{H}^2(\omega).
\end{equation}
From this formula we deduce that $W\ast\chi\in C^1(\mathbb{R}^3)$ and
that for any $x\in \R^3$
\begin{equation}\label{potgradbis}
\nabla(W\ast\chi)(x)\cdot x
=-\frac{1}{|B|}\sqrt{\frac{\pi}2}
\int_{\mathbb{S}^2}\alpha^2(x,\omega)\chi_{(-1,1)}(\alpha(x,\omega))\frac{\widehat{\Psi}(\omega)}{{|D(a)R^T\omega|}}\,d\mathcal{H}^2(\omega).
\end{equation}

Finally, we compute the Hessian of $W\ast\chi$ in the interior of $E$. 
From \eqref{potgradbis-0}, for any $j,k=1,2,3$ and any $x$ in the interior of $E$, we have  that 
\begin{equation}\label{pothes}
\frac{\partial^2(W\ast\chi)}{\partial x_j\partial x_k}(x)
=-\frac{1}{|B|}\sqrt{\frac{\pi}2}
\int_{\mathbb{S}^2}\frac{\omega_j\omega_k\widehat{\Psi}(\omega)}{|D(a)R^T\omega|^3}\,d\mathcal{H}^2(\omega),
\end{equation}
where we used that $\chi_{(-1,1)}(\alpha(x,\omega))=1$ for every $\omega\in\mathbb{S}^2$.

\section{Main results}\label{sec:main-res}
In this section we state and prove the main result of the paper. We start by introducing some notation. 

For any ellipsoid $E$ as in \eqref{ellrot}, we define the corresponding \emph{ellipsoid law} as the measure
\begin{equation}\label{ellipsoidlaw}
\mu_E(x):=\frac{\chi_E(x)}{|E|}=\frac{\chi_{E_0(a)}(R^Tx)}{|E_0(a)|}=\mu_{E_0(a)}(R^Tx),
\end{equation}
namely, as the normalised characteristic function of the ellipsoid.

For an ellipse $\tilde{E}$ as in \eqref{ellipserot}, we define the corresponding \emph{semi-ellipsoid law} as the measure
\begin{equation}\label{ellipselaw}
\mu_{\tilde{E}}(x):=\mu_{\tilde{E}_0(a_1,a_2)}(R^Tx),
\end{equation}
where we set
$$
\mu_{\tilde{E}_0(a_1,a_2)}(x):=\frac{2}{|B|a_1a_2}\sqrt{1-\frac{x_1^2}{a_1^2}-\frac{x_2^2}{a_2^2}}\,  \mathcal{H}^2\mres \tilde{E}_0(a_1,a_2)\otimes \delta_0(x_3).
$$

Our main result is the following theorem.

\begin{theorem}\label{Mainthm}
Let $W$ be as in \eqref{potdef} with $\Psi$ an even and strictly positive function in $H^s(\mathbb{S}^2)$, for $s>\frac32$. Assume that $\widehat{W}\geq 0$ on $\mathbb{S}^2$.
Then there exists a unique minimiser $\mu_0\in \mathcal{P}(\R^3)$ of $I$, which can be characterised as follows.

\begin{enumerate}[\textnormal{(}a\textnormal{)}]
\item\label{nondegenerate} If $\widehat{W}> 0$ on $\mathbb{S}^2$,
then $\mu_0$ is an ellipsoid law of the form \eqref{ellipsoidlaw}.

\item\label{degenerate} If $\widehat{W}\geq 0$ on $\mathbb{S}^2$, then
either $\mu_0$ is an ellipsoid law of the form \eqref{ellipsoidlaw} or $\mu_0$ is a
semi-ellipsoid law of the form \eqref{ellipselaw}.
\end{enumerate}
\end{theorem}

\begin{remark}\label{shrinkrem} Note that both cases in Theorem~\ref{Mainthm}($\ref{degenerate}$) can indeed occur. For instance, for the potential 
$$W(x)=\frac{1}{|x|}+\frac{x_1^2}{|x|^3},\quad x\in \R^3,\ x\neq 0,$$
whose Fourier transform is 
$$\widehat{W}(\xi)=\sqrt{\frac{8}{\pi}}\frac{(1-(\xi_1/|\xi|)^2)}{|\xi|^2},$$ 
it is shown in  \cite{CMMRSV2} that the minimiser of $I$ is the normalised characteristic function of a non-degenerate ellipsoid, hence an ellipsoid law as in  \eqref{ellipsoidlaw}. Instead, in Example ~\ref{ex:loss-dim} below, we construct a profile for which loss of dimension occurs.
\end{remark}

We proceed as follows. In Section~\ref{sect:first} we prove existence and uniqueness of the minimiser of the energy $I$, and show that it is characterised by the Euler-Lagrange conditions for $I$. In Section~\ref{strict:FT} we prove that there exists a minimising ellipsoid under the additional assumption that $\widehat{W}$ is strictly positive on $\mathbb{S}^2$, that is, we prove Theorem~\ref{Mainthm}($\ref{nondegenerate}$).
In Section~\ref{lossdimsec}, we consider the degenerate case and prove Theorem~\ref{Mainthm}($\ref{degenerate}$). Moreover, in Example~\ref{ex:loss-dim} we exhibit a potential for which loss of dimension occurs.

\subsection{Existence, uniqueness and characterisation of the minimiser}\label{sect:first}

We have the following result.
\begin{prop}
Let $W$ be as in \eqref{potdef} with $\Psi$ even, strictly positive, and in $H^s(\mathbb{S}^2)$, for $s>\frac32$. Assume that $\widehat{W}\geq 0$ on $\mathbb{S}^2$.
Then there exists a unique minimiser $\mu_0$ of $I$, which is characterised by the following Euler-Lagrange conditions:
\begin{align}\label{EL1}
\left(W\ast \mu_0  \right) (x)+ \frac{1}{2}|x|^2&= C \quad \text{for }\mu_0\text{-a.e. } x \in \supp\mu_0,\\
\left(W\ast \mu_0  \right) (x)+ \frac{1}{2}|x|^2&\ge C \quad \text{for }x \in \R^3 \setminus N \text{ with } \operatorname{Cap}(N)=0, \label{EL2}
\end{align}
where $\supp\mu_0$ stands for the support of $\mu_0$, $C$ is a constant, and $\operatorname{Cap}$ is the capacity.
\end{prop}

\begin{proof} Arguing as in \cite[Proposition~2.1]{CMMRSV2}, one can prove that a minimiser exists and that any minimiser has compact support and finite energy. We also note that a bound on the compact support can be found depending only on the maximum of $\Psi$ on $\mathbb{S}^2$.
Moreover, the energy $I$ is strictly convex on the subset of $\mathcal{P}(\R^3)$ given by
$$
\widetilde{\mathcal{P}}:=\Big\{\nu\in \mathcal{P}(\R^3):\ \nu\text{ has compact support and }\int_{\R^3}(W\ast\nu)(x)\,d\nu(x)<+\infty\Big\}.
$$
Strict convexity immediately implies uniqueness of the minimiser. The characterisation of the minimiser through the Euler-Lagrange conditions can then be deduced as in 
\cite[Theorem~3.1]{MRS}.

We briefly comment on the proof of strict convexity of $I$ on $\widetilde{\mathcal{P}}$. Arguing as in \cite[Proposition~2.1]{CMMRSV2}, strict convexity is equivalent to 
the following inequality: for any $\nu_1,\nu_2\in \widetilde{\mathcal{P}}$ 
\begin{equation}\label{F>0}
\int_{\R^3}\widehat{W}(\xi)|\widehat{(\nu_1-\nu_2)}(\xi)|^2\,d\xi\geq0
\end{equation}
and equality holds if and only if $\nu_1=\nu_2$. The inequality \eqref{F>0} follows immediately from \eqref{hatW} and the assumption $\widehat{W}\geq 0$ on $\mathbb{S}^2$. To show that equality holds in \eqref{F>0} if and only if $\nu:=\nu_1-\nu_2=0$, one can follow the proof of \cite[Theorem~2.6]{CS2d}. In fact, since by assumption $W>0$, its Fourier transform $\widehat{W}$ cannot be identically $0$. Hence, there exist $c_0>0$, $x^0\in \R^3$, and $r>0$, with $r<|x^0|$, such that
$\widehat{W}\geq c_0$ in $B_r(x^0)$. If
$$
\int_{\R^3}\widehat{W}(\xi)|\widehat{\nu}(\xi)|^2\,d\xi=0,
$$
then $\widehat{\nu}=0$ on $B_r(x^0)$. Since by the Paley-Wiener Theorem $\widehat{\nu}$ is the restriction
of an entire function to $\mathbb{R}^3$, $\widehat{\nu}$ must be identically zero. This concludes the proof.
\end{proof}

\subsection{The case of strictly positive Fourier transform}\label{strict:FT}
In this section we assume that $\widehat{W}$ is strictly positive on $\mathbb{S}^2$ and we prove Theorem~\ref{Mainthm}($\ref{nondegenerate}$).

\begin{proof}[Proof of Theorem~\textnormal{\ref{Mainthm}($\ref{nondegenerate}$)}]
Let $E$ be an ellipsoid of the form \eqref{ellrot} and let $\chi:=\chi_E/|E|$. For any $x\in \mathbb{R}^3$ we define the \emph{potential} \begin{equation}\label{potential_EL}
P(x):=(W\ast\chi)(x)+\frac{|x|^2}2. 
\end{equation}
We need to show that there exists an ellipsoid $E$
such that the corresponding function $P$ in \eqref{potential_EL} satisfies \eqref{EL1} and \eqref{EL2}.

As a preliminary remark, we note that a direct consequence of formula \eqref{pot}, together with the definition of $\alpha$ in \eqref{alpha-rho}, is that
\begin{equation*}
P(x)=P(0)+P_2(x)\quad\text{for any }x\in E,
\end{equation*}
where $P_2$ is a homogeneous second-order polynomial, whose expression can be computed in terms of $\widehat{\Psi}$ by \eqref{pothes}.

We divide the proof into three steps.\smallskip

\noindent
\textit{Step 1: The first Euler-Lagrange condition.} By the regularity and evenness of the potential $P$ in \eqref{potential_EL}, proving condition \eqref{EL1} is equivalent to showing that the Hessian of $P$ vanishes on $E$. By \eqref{pothes} this is equivalent to showing that for $i,j=1,2,3$
\begin{equation}\label{EL1-Fourier}
\frac{1}{|B|}\sqrt{\frac{\pi}2}
\int_{\mathbb{S}^2}\frac{\omega_i\omega_j\widehat{\Psi}(\omega)}{|D(a)R^T\omega|^3}\,d\mathcal{H}^2(\omega)=\delta_{ij},
\end{equation}
$\delta_{ij}$ being the Kronecker delta.

We need to show that
 there exist $a=(a_1,a_2,a_3)\in \R^3$, with $a_i>0$ for $i=1,2,3$, and $R\in SO(3)$ such that $E$ as in \eqref{ellrot} satisfies \eqref{EL1-Fourier}.
We begin by showing that \eqref{EL1-Fourier} can be written as a stationarity condition for an auxiliary scalar function defined on $\mathbb{M}_+$,
where
$\mathbb{M}_+$ is the space of positive definite symmetric $3\times 3$ matrices. 

Note that 
\begin{align*}
|D(a)R^T\omega|^3 &= (D(a)R^T\omega\cdot D(a)R^T\omega)^{3/2}=((D(a)R^T)^T D(a)R^T\omega\cdot \omega)^{3/2}\\
&=(R(D(a))^2R^T\omega\cdot \omega)^{3/2} = (R D(A)R^T\omega\cdot \omega)^{3/2},
\end{align*}
where $A=(A_1,A_2,A_3)$ is such that $A_i=a_i^2$ for $i=1,2,3$. By setting $M:= R D(A)R^T$, we observe that $M\in \mathbb{M}_+$. In terms of $M$, condition \eqref{EL1-Fourier} reads as
\begin{equation}\label{EL1-FourierM} 
\frac{1}{|B|}\sqrt{\frac{\pi}2}\int_{\S^2} \frac{\omega_i\omega_j\widehat \Psi(\omega)}{(M\omega\cdot \omega)^{3/2}} 
\, d\HH(\omega) = \delta_{ij},  \quad \text{for } i,j =1,2,3.
\end{equation}

We define the function $f: \mathbb{M}_+ \to \R$ as
\begin{equation}\label{def:function-f}
f(M):=g(M)+h(M)=\frac{\sqrt{2\pi}}{|B|}\int_{\mathbb{S}^2}\frac{\widehat{\Psi}(\omega)}{\sqrt{M\omega\cdot \omega}}d\HH(\omega)+\mathrm{tr}(M),
\end{equation}
where $M=[M_{ij}]_{i,j=1}^3$. Note that $f, g\geq 0$. 
We observe that the first Euler-Lagrange condition is satisfied if there exists $M_0\in\mathbb{M}_+$ such that $\nabla_M f(M_0)=0$, 
where we denoted, with some abuse of notation,
$$\nabla_M=\left(\frac{\partial}{\partial M_{11}},\frac{\partial}{\partial M_{22}},\frac{\partial}{\partial M_{33}},\frac{\partial}{\partial M_{12}},\frac{\partial}{\partial M_{13}},\frac{\partial}{\partial M_{23}}\right).$$
Therefore, the proof of \eqref{EL1-Fourier} is concluded if we show that there exists a stationary point for $f$. 
In the next step we actually show that $f$ has a global minimiser $M_0$ on $\mathbb{M}_+$. Since $\mathbb{M}_+$ is an open set in $\R^6$, this will imply that $M_0$ is a solution of the system \eqref{EL1-FourierM} and the proof is concluded.\smallskip

\noindent
\textit{Step 2: The function $f$ in \eqref{def:function-f} has a global minimiser on $\mathbb{M}_+$.}
We rewrite the minimisation problem $\min_{M\in \mathbb{M}_+}f(M)$ in a simpler form. 
We start by analysing the behaviour of $f$ along lines passing through the origin. For every fixed line we characterise the minimum point of $f$ on the line, and we then minimise over the lines. 

More precisely, let $M\in \mathbb{M}_+$ be fixed, and let $F: (0,+\infty)\to \R$ be the function defined as $F(t):=f(tM)$. Note that 
\begin{equation*}
F(t)=\frac{1}{\sqrt{t}}g(M)+th(M),
\end{equation*}
where we have used the scaling properties of the terms $g$ and $h$ in $f$. Hence the function $F$ has a unique minimiser at
$$t_{\text{min}}(M)=\left(\frac{g(M)}{2h(M)}\right)^{2/3}.$$
We now define the function $\tilde f: \mathbb{M}_+ \to \R$ as 
$$\tilde{f}(M):=f(t_{\text{min}}(M)M)=\frac{3}{2^{2/3}}g(M)^{2/3}h(M)^{1/3},$$ 
and we note that it is constant on rays emanating from the origin. Therefore, in minimising $\tilde{f}$ we can
introduce the constraint $\mathrm{tr}(M)=1$. Moreover, under this constraint, the problem 
$$\min_{M\in\mathbb{M}_+:\ \mathrm{tr}(M)=1}\tilde{f}(M)$$
is equivalent to
\begin{equation}\label{critmin1}
\min_{M\in\mathbb{M}_+:\ \mathrm{tr}(M)=1}g(M).
\end{equation}

We now consider the sets
$$T=\left\{A=(A_1,A_2,A_3)\in \R^3:\ A_i>0\text{ and }\sum_{i=1}^3A_i=1\right\}\quad\text{and}\quad\Omega=T\times SO(3),$$
endowed with the usual Euclidean distance.
For any $(A,Q)\in \Omega$, we call 
$M(A,Q)=Q^TD(A)Q$. This map is surjective onto the set of matrices $M\in \mathbb{M}_+$ with $\mathrm{tr}(M)=1$.
Thus \eqref{critmin1} is in turn equivalent to 
$$\min_{(A,Q)\in \Omega}\int_{\mathbb{S}^2}\frac{\widehat{\Psi}(\omega)}{\sqrt{ D(A)Q\omega\cdot Q\omega}}\,d\HH(\omega),$$
where 
we have neglected the positive constant factor in the definition of $g$.

By a change of variables, the minimisation problem we want to solve is  
$$\min_{(A,Q)\in \Omega}\gamma (A,Q),$$
where 
$$
\gamma(A,Q):=\int_{\mathbb{S}^2}\frac{\widehat{\Psi}(Q^T\omega)}{\sqrt{ D(A)\omega\cdot \omega}}\,d\HH(\omega).
$$
Clearly the function $\gamma$ is positive and continuous on $\Omega$. Moreover, since 
\begin{equation}\label{bound:What}
0<C_0\leq \widehat{\Psi}(\omega)\leq C_1 \quad \text{for all } \omega\in  \mathbb{S}^2,
\end{equation}
for some positive constants $C_0$ and $C_1$, the function $\gamma$ satisfies the bounds
\begin{equation}\label{bound:gamma}
C_0\int_{\mathbb{S}^2}\frac{1}{\sqrt{D(A)\omega\cdot \omega}}\,d\HH(\omega)  \leq \gamma(A,Q)\leq C_1 \int_{\mathbb{S}^2}\frac{1}{\sqrt{D(A)\omega\cdot \omega}}\,d\HH(\omega)
\end{equation}
for any $(A,Q)\in \Omega$.
Note that, since $SO(3)$ is compact, 
$$\overline{\Omega}=\overline{T}\times SO(3)$$ is a compact set,
where clearly
$$\overline{T}=\left\{A=(A_1,A_2,A_3)\in \R^3:\ A_i\geq 0\text{ and }\sum_{i=1}^3A_i=1\right\}.$$
Hence, if $\gamma$ could be extended by continuity to $\overline{\Omega}$, we would directly derive the existence of a global minimiser of $\gamma$ in $\overline \Omega$, and the remaining step to be proved would be that such a minimiser is, in fact, in $\Omega$. This is our strategy of proof.\smallskip

\noindent
\textit{Step 2.1: The function $\gamma$ has a global minimiser in $\overline{\Omega}$.} 
We observe that 
$$
\overline{T} \setminus T = \mathcal{V} \cup \mathcal{E},
$$
where $\mathcal{V}=\{e_1,e_2,e_3\}$ and 
$$\mathcal{E}=\left\{(\sigma,1-\sigma,0), (\sigma,0,1-\sigma), (0,\sigma,1-\sigma):\ 0<\sigma<1
\right\}.$$
We extend $\gamma$ to $\mathcal{V}\times SO(3)$ first, and then to $\mathcal{E}\times SO(3)$. 
Let $(A^0,Q)\in \mathcal{V}\times SO(3)$. We set $\gamma(A^0,Q):=+\infty$ and we now show that 
\begin{equation}\label{goal:1}
\lim_{\substack{A\in T\\ A\to A^0}} \gamma(A,Q)=+\infty \quad \text{uniformly in } Q.
\end{equation}
With no loss of generality we can assume that $A^0=e_3$, so that we can restrict our attention to $A=(A_1,A_2,A_3) \in T$ with $0<A_1,A_2<\delta<1/4$.
Then by \eqref{bound:gamma}
$$\gamma(A,Q)\geq C_0\int_{\mathbb{S}^2}\frac{1}{\sqrt{\delta (\omega_1^2+\omega_2^2)+\omega_3^2}}\,d\HH(\omega),$$
where the right hand side goes to $+\infty$, as $\delta$ goes to $0$, which proves \eqref{goal:1}.

Now, let $(A^0,Q^0)\in \mathcal{E}\times SO(3)$. With no loss of generality we can assume that $A^0=(A^0_1,A^0_2,0)$ and $A^0_1,A^0_2>0$. 
Let $(A^n,Q^n)\in \Omega$ be such that $(A^n,Q^n)\to (A^0,Q^0)$ as $n\to\infty$.
We define 
$$\gamma(A^0,Q^0):=\int_{\mathbb{S}^2}\frac{\widehat{\Psi}((Q^0)^T\omega)}{\sqrt{A^0_1\omega_1^2+A^0_2\omega_2^2}}\,d\HH(\omega)<+\infty, $$
and we show that 
\begin{equation}\label{limlim}
\lim_n\gamma(A^n,Q^n)=\gamma(A^0,Q^0).
\end{equation}
Note that, for $\omega\in \mathbb{S}^2$, we have that 
$$
\frac{\widehat{\Psi}((Q^n)^T\omega)}{\sqrt{ D(A^n)\omega\cdot \omega}}\to \frac{\widehat{\Psi}((Q^0)^T\omega)}{\sqrt{A^0_1\omega_1^2+A^0_2\omega_2^2}}
$$
as $n\to\infty$. Let $0<a_0<\min\{A^0_1,A^0_2\}$. Since $A^n\to (A^0_1,A^0_2,0)$, there exists $n_0\in \mathbb{N}$ such that $A^n_1,A^n_2>a_0$
for any $n\geq n_0$, and so
$$
\frac{\widehat{\Psi}((Q^n)^T\omega)}{\sqrt{ D(A^n)\omega\cdot \omega}}\leq \frac{C_1}{\sqrt{a_0\omega_1^2+a_0\omega_2^2}} \in L^1(\mathbb{S}^2).
$$
Hence by the Dominated Convergence Theorem we deduce \eqref{limlim}.

The existence of a minimiser of $\gamma$ on $\overline{\Omega}$ is now established.\smallskip

\noindent
\textit{Step 2.2: The function $\gamma$ attains a global minimiser in $\Omega$.} 
Assume by contradiction that $(A^0,Q^0)\in \partial T\times SO(3)$ is a global minimiser. Clearly, $A^0\in \mathcal{E}$, thus we can assume
without loss of generality that $A^0=(A^0_1,A^0_2,0)$ with $A^0_1,A^0_2>0$.
To reach a contradiction we show that $\gamma$ decreases when moving from $A^0$ towards the interior of $T$, at least close to $A^0$.

Let, as before, $0<a_0<\min\{A^0_1,A^0_2\}$, and for any $0<r<a_0$ let $A^r:=(A^0_1-r/2,A^0_2-r/2,r)$. Note that $A^r\in T$.

We claim that for some $0<\delta_0<a_0$, we have
\begin{equation}\label{claim:decreasing}
\frac{d}{dr}\gamma(A^r,Q^0)<0\quad\text{for any }0<r<\delta_0,
\end{equation}
which contradicts the minimality of $(A^0,Q^0)$, therefore proving that $\gamma$ attains a global minimiser in $\Omega$.
Let $0<r<a_0$. We have
\begin{eqnarray*}
\frac{d}{dr}\gamma(A^r,Q^0) & = & -\frac12\int_{\mathbb{S}^2}\frac{\widehat{\Psi}(Q^T\omega)(-\omega_1^2/2-\omega_2^2/2+\omega_3^2)}{(D(A^r)\omega\cdot \omega)^{3/2}}\,d\HH(\omega)
\\
& = &
-\frac14\int_{\mathbb{S}^2}\frac{\widehat{\Psi}(Q^T\omega)(2-3(\omega_1^2+\omega_2^2))}{(D(A^r)\omega\cdot \omega)^{3/2}}\,d\HH(\omega)
\\
& = &
-\frac12\int_{\mathbb{S}^2}\frac{\widehat{\Psi}(Q^T\omega)}{(D(A^r)\omega\cdot \omega)^{3/2}}\,d\HH(\omega)+\frac34
\int_{\mathbb{S}^2}\frac{\widehat{\Psi}(Q^T\omega)(\omega_1^2+\omega_2^2)}{(D(A^r)\omega\cdot \omega)^{3/2}}\,d\HH(\omega)
\\
& =: & -\beta_1(r)+\beta_2(r).
\end{eqnarray*}
Note that $\beta_1(r), \beta_2(r)>0$. To prove \eqref{claim:decreasing} we show that $\beta_1$ dominates $\beta_2$ for $r$ small enough. 

Using spherical coordinates, the bounds \eqref{bound:What}, and the fact that $A^0_1,A^0_2<1$, we have
$$\beta_1(r)\geq \frac{C_0}2\int_{0}^{2\pi}\int_0^{\pi} \frac{\sin\psi}{(\sin^2\psi+r\cos^2\psi)^{3/2}}\,d\psi\,d\theta.
$$
As $r\to 0^+$, the right hand side goes to $+\infty$, namely
$$\lim_{r\to 0^+}\frac{C_0}2\int_{0}^{2\pi}\int_0^{\pi} \frac{\sin\psi}{(\sin^2\psi+r\cos^2\psi)^{3/2}}\,d\psi\,d\theta=\frac{C_0}2\int_{0}^{2\pi}\int_0^{\pi} \frac1{\sin^2\psi}\,d\psi\,d\theta=+\infty.$$
Hence
\begin{equation}\label{alphaeq}
\lim_{r\to 0^+}\beta_1(r)=+\infty.
\end{equation}
Instead, by the definition of $a_0$ and since $r<a_0$, we have
\begin{equation}\label{betaeq}
0<\beta_2(r)\leq\frac34
\int_{0}^{2\pi}\int_0^{\pi}\frac{C_1\sin^3\psi}{((a_0/2)\sin^2\psi+r\cos^2\psi)^{3/2}}\,d\psi\,d\theta\leq \frac{3\pi^2C_1}{2(a_0/2)^{3/2}}.
\end{equation}
The claim \eqref{claim:decreasing} easily follows by \eqref{alphaeq} and \eqref{betaeq}, and the proof is concluded.\smallskip

\noindent
\textit{Step 3: The first Euler-Lagrange condition implies the second one.}
To conclude we show that, if an ellipsoid $E$ satisfies \eqref{EL1}, then it also satisfies \eqref{EL2}.
In fact, assume that $E$ satisfies \eqref{EL1}. It is enough to show that for any $x\in \R^3$ we have
\begin{equation*}
\nabla P(x)\cdot x=
 \nabla (W\ast \chi) (x)\cdot x +  |x|^2 \geq 0,
\end{equation*}
for $P$ defined in \eqref{potential_EL}.
By \eqref{EL1} this property is obvious for $x\in E$, therefore it is enough to prove it for $x \in \R^3\setminus E$.

Let $x\in \R^3\setminus E$. We write $x=tx^0$ with $t>1$ and $x^0\in\partial E$. By \eqref{potgradbis} and \eqref{EL1}
we know that
\begin{equation}\label{eq10}
-\frac{1}{|B|}\sqrt{\frac{\pi}2}
\int_{\mathbb{S}^2}\alpha^2(x^0,\omega)\frac{\widehat{\Psi}(\omega)}{{|D(a)R^T\omega|}}\,d\mathcal{H}^2(\omega)+|x^0|^2=0,
\end{equation}
where we have used that for $x^0\in\partial E$ we have $|\alpha(x^0,\omega)|<1$ except for $\omega$ on a set of $\mathcal{H}^{2}$-measure zero.
By \eqref{potgradbis} and \eqref{eq10} we deduce that
\begin{align*}
\nabla P(x)\cdot x&  =  -\frac{1}{|B|}\sqrt{\frac{\pi}2}
\int_{\mathbb{S}^2}\alpha^2(x,\omega)\chi_{(-1,1)}(\alpha(x,\omega))\frac{\widehat{\Psi}(\omega)}{{|D(a)R^T\omega|}}\,d\mathcal{H}^2(\omega)+|x|^2\\
& = 
t^2\left[-\frac{1}{|B|}\sqrt{\frac{\pi}2}
\int_{\mathbb{S}^2}\alpha^2(x^0,\omega)\chi_{(-1,1)}(\alpha(x,\omega))\frac{\widehat{\Psi}(\omega)}{{|D(a)R^T\omega|}}\,d\mathcal{H}^2(\omega)+|x^0|^2\right]\\
 &=  t^2\left[\frac{1}{|B|}\sqrt{\frac{\pi}2}
\int_{\mathbb{S}^2}\alpha^2(x^0,\omega)\big(1-\chi_{(-1,1)}(\alpha(x,\omega))\big)\frac{\widehat{\Psi}(\omega)}{{|D(a)R^T\omega|}}\,d\mathcal{H}^2(\omega)\right]\geq 0.
\end{align*}
The proof is complete.
We point out that Step 3 works under the weaker assumption that $\widehat{W}\geq 0$ on $\mathbb{S}^2$.
\end{proof}

\subsection{The loss of dimension in the degenerate case}\label{lossdimsec}

In this section we prove Theorem~\ref{Mainthm}($\ref{degenerate}$) and show that loss of dimension of the minimiser can occur, as claimed in  Remark~\ref{shrinkrem}. 

Let us introduce the following further notation. For a given constant $a_1>0$, we set $S_0(a_1):=[-a_1,a_1]\times\{0\}\times\{0\}$.
A general one-dimensional segment $S$ centred at the origin
can be then obtained by rotating $S_0(a_1)$ with respect to the coordinate axes, namely as
\begin{equation}\label{segmentrot}
S= RS_0(a_1),
\end{equation}
for some rotation $R\in SO(3)$.

We associate to $S_0$ the probability measure
$$\mu_{S_0(a_1)}:= \frac{\pi}{|B|a_1} \left(1-\frac{x_1^2}{a_1^2}\right)\mathcal{H}^1\mres S_0(a_1)\otimes \delta_0(x_2,x_3),$$
and to any segment $S$ as in \eqref{segmentrot} the corresponding measure
\begin{equation}\label{segmentlaw}
\mu_{S}(x):=\mu_{S_0(a_1)}(R^Tx).
\end{equation}

\begin{proof}[Proof of Theorem~\textnormal{\ref{Mainthm}($\ref{degenerate}$)}]
Let $W_0$ be a potential satisfying the assumptions of Theorem~\ref{Mainthm}, with $\widehat{W_0}\geq 0$ on $\mathbb{S}^2$, but not strictly positive.
We denote its profile by $\Psi_0$ and the function defined in \eqref{hatW} by $\widehat{\Psi_0}$. We denote by $I_0$ the energy defined as in \eqref{en:general}, corresponding to the potential $W_0$.

For $\varepsilon>0$, we set
\begin{equation}\label{pert-psi}
\Psi_{\varepsilon}:=\Psi_{0}+\sqrt{\frac{\pi}2}\varepsilon.
\end{equation} 
We call $W_{\varepsilon}$ the corresponding potential as in \eqref{potdef}, and $I_{\varepsilon}$ the corresponding energy. Note that $W_\varepsilon>0$ and that its Fourier transform on $\mathbb{S}^2$ is 
$$
\widehat{\Psi_\varepsilon} 
=\widehat{\Psi_0} + \varepsilon >0,
$$
where $\widehat{\Psi_\varepsilon}$ is defined as in \eqref{hatW}. 
By Theorem~\ref{Mainthm}($\ref{nondegenerate}$) there exists an ellipsoid, that we denote by $E^{\varepsilon}$, such that $\mu_{\varepsilon}=\mu_{E^{\varepsilon}}$ minimises the energy $I_{\varepsilon}$.
As in \eqref{ellrot}, we write $E^{\varepsilon}=R^{\varepsilon}D(a^{\varepsilon})\overline{B}$, where $R^{\varepsilon}\in SO(3)$ and $a^{\varepsilon}=(a_1^{\varepsilon},a_2^{\varepsilon},a_3^{\varepsilon})\in\R^3$ with
$a_i^{\varepsilon}>0$.

By compactness, up to subsequences, we have that, as $\varepsilon\to 0^+$,
$a^{\varepsilon}\to a^0\in [0,+\infty)^{3}$. Depending on the number of zero components of $a^0$ we have the following four cases.
\begin{enumerate}[(1)]
\item $a^{0}=(0,0,0)$. In this case,  as $\varepsilon\to 0^+$, $\mu_{\varepsilon}$ converges narrowly to $\mu_0=\delta_0$.
\item Up to a permutation of the variables, $a^0=(a^0_1,0,0)$ with $a^0_1>0$. In this case 
there exists a segment $S$ as in \eqref{segmentrot} such that,  as $\varepsilon\to 0^+$,
$\mu_{\varepsilon}$ converges narrowly to $\mu_0=\mu_S$ as in \eqref{segmentlaw}.

\item Up to a permutation of the variables, $a^0=(a^0_1,a^0_2,0)$ with $a^0_1,a^0_2>0$. In this case
there exists an ellipse $\tilde{E}$ as in \eqref{ellipserot} such that,  as $\varepsilon\to 0^+$,
$\mu_{\varepsilon}$ converges narrowly to $\mu_0=\mu_{\tilde{E}}$ as in \eqref{ellipselaw}.

\item $a^0=(a^0_1,a^0_2,a^0_3)$ with $a^0_1,a^0_2,a^0_3>0$. In this case, 
there exists an ellipsoid $E$ as in \eqref{ellrot} such that,  as $\varepsilon\to 0^+$, 
$\mu_{\varepsilon}$ converges to $\mu_0=\mu_E=\chi_{E}/|E|$
 as in \eqref{ellipsoidlaw}.
\end{enumerate}

It is easy to see that
$$I_0(\mu_0)\leq \liminf_{\varepsilon\to 0^+}I_{\varepsilon}(\mu_{\varepsilon})$$
and that
$$I_0(\mu)=\lim_{\varepsilon\to 0^+}I_{\varepsilon}(\mu)\quad\text{for any }\mu\in\mathcal{P}(\R^3).$$
Hence $\mu_0$ is the unique minimiser of $I_0$.
In order to conclude the proof, it remains to show that cases (1) and (2) cannot occur. This can be readily seen by showing that, in both cases,
$I_0(\mu_0)=+\infty$. This is trivial for case (1), since
$I_0(\delta_0)=W_0(0)=+\infty$. For case (2), it is enough to prove that
$I_{\textrm{C}}(\mu_{S_0(a_1)})=+\infty$ for any $a_1>0$, where $I_{\textrm{C}}$ is the Coulomb energy corresponding to $W_{\textrm{C}}(x)=1/|x|$. In fact,  $I_0$ is bounded from below by a positive multiple of $I_{\textrm{C}}$. 

For the Coulomb energy we have
$$
I_{\textrm{C}}(\mu_{S_0(a_1)})\geq \int_{-a_1}^{a_1}\left(\int_{\R^3} \frac{1}{|x_1e_1-y|}d\mu_{S_0(a_1)}(y)\right)\frac{\pi}{|B|a_1} \left(1-\frac{x_1^2}{a_1^2}\right)dx_1
=+\infty,
$$
where we have used that for any $|x_1|<a_1$
$$\int_{\R^3} \frac{1}{|x_1e_1 -y|}d\mu_{S_0(a_1)}(y)=\int_{-a_1}^{a_1} \frac{1}{|x_1-y_1|}\frac{\pi}{|B|a_1} \left(1-\frac{y_1^2}{a_1^2}\right)   dy_1=+\infty.$$
This completes the proof.
\end{proof}

We now provide an explicit example where case (3) in the proof of Theorem~\textnormal{\ref{Mainthm}($\ref{degenerate}$)} occurs.

\begin{example}[Loss of dimension for the minimiser]\label{ex:loss-dim}
Let us consider the profile 
$$
\Psi_0(x_1,x_2,x_3)= \frac{1}{2} \sqrt{\frac{\pi}{2}} \big(4+3 x_3^2 + 3 x_3^4\big), \quad x \in  \mathbb{S}^2.
$$
We will see below in Remark \ref{rem:quad} that loss of dimensionality cannot occur for quadratic profiles, hence a profile providing an example of a degenerate minimiser has to be a polynomial of degree at least four.

Clearly $\Psi_0$ is even and strictly positive. To compute $\widehat{\Psi_0}$, we express $\Psi_0$ in terms of spherical harmonics.
The homogeneous polynomials
\begin{equation*}
P_0(x) =1, \quad\quad P_2(x)=3x_3^2 -|x|^2, \quad \quad P_4(x)=35x_3^4 -30x_3^2 |x|^2 +3 |x|^4 
\end{equation*}
are harmonic and so, being of different degrees,  their restrictions to the unit sphere
\begin{equation*}
P_0(x) =1, \quad\quad\quad P_2(x)=3x_3^2 -1,  \quad \quad   P_4(x)=35x_3^4 -30x_3^2  +3 
\end{equation*}
are orthogonal spherical harmonics. It is easily seen that $\Psi_0$ is a linear combination of $P_0$, $P_2$, and $P_4$. More precisely,
$$
\Psi_0(x) = \sqrt{\frac{\pi}{2}}\frac{1}{35} \left( 98 +\frac{65}{2} P_2(x)  +\frac{3}{2}P_4(x) \right), \quad x \in  \mathbb{S}^2.
$$
The polynomials $P_0, P_2, P_4$ are eigenvectors of the mapping $\Psi \mapsto \widehat{\Psi}$
with eigenvalues $\sqrt{\frac{2}{\pi}}, -2 \sqrt{\frac{2}{\pi}}, \frac{8}{3} \sqrt{\frac{2}{\pi}}$, respectively. Hence
\begin{eqnarray*}
\widehat{\Psi_0}(\xi)& = &\frac{1}{35} \left(98 -65 P_2(\xi) + 4 P_4(\xi)\right)
\\
&=&
5-9\xi_3^2+4\xi_3^4=(1-\xi_3^2)+4(1-\xi_3^2)^2=(\xi_1^2+\xi_2^2)+4(\xi_1^2+\xi_2^2)^2, \quad \xi\in \mathbb{S}^2,
\end{eqnarray*}
which is nonnegative on $\mathbb{S}^2$, but not strictly positive. Hence the potential $W_0$ associated with the profile $\Psi_0$ via \eqref{potdef} satisfies the assumptions of Theorem~\textnormal{\ref{Mainthm}($\ref{degenerate}$)}.

Let $I_0$ be the energy corresponding to the potential $W_0$. We claim that the minimiser of $I_0$ is not the normalised characteristic function of an ellipsoid and so, by Theorem~\ref{Mainthm}, it is of the form \eqref{ellipselaw}.  

First of all, if the minimiser were of the form $\chi_{E}/|E|$, with $E$ an ellipsoid as in \eqref{ellrot}, then by symmetry and uniqueness 
$E$ would be a spheroid of equation
$$
\frac{x_1^2}{a^2_0}+\frac{x_2^2}{a^2_0}+\frac{x_3^2}{b^2_0}\leq 1
$$
for some $a_0>0$ and $b_0>0$. Moreover, from the first Euler-Lagrange condition for minimality for $\chi_{E}/|E|$, $a_0$ and $b_0$ would satisfy the system of equations \eqref{EL1-FourierM}, with $M=D(a_0^2,a_0^2,b_0^2)$, and in particular
$$
\int_{\S^2} \frac{(\omega_1^2+\omega_2^2-2\omega_3^2)\widehat{\Psi_0}(\omega)}{(M\omega\cdot \omega)^{3/2}} 
\, d\HH(\omega) = 0.
$$
Using spherical coordinates and the expression of $\widehat{\Psi_0}$, this condition can be written as
$$
0=\int_0^\pi\frac{(\sin^2\psi-2\cos^2\psi)(\sin^2\psi+4\sin^4\psi)}{(a_0^2\sin^2\psi+b_0^2\cos^2\psi)^{3/2}} 
\sin\psi\, d\psi = \frac{p(t_\ast)}{a_0^3},
$$
where $t_\ast:=b_0^2/a_0^2$ and the function $p:[0,+\infty)\to\R$ is defined as
$$
p(t):=\int_0^\pi\frac{(\sin^2\psi-2\cos^2\psi)(\sin^2\psi+4\sin^4\psi)}{(\sin^2\psi+t\cos^2\psi)^{3/2}} 
\sin\psi\, d\psi
$$
for every $t\geq0$. 

We now show that $p(t)\neq0$ for every $t>0$, which implies that the minimiser of $I_0$ is not the normalised characteristic function of an ellipsoid. 

To this aim, we first note that $p$ is a continuous function in $[0,+\infty)$ and
$$
p(0)=\int_0^\pi (\sin^2\psi-2\cos^2\psi)(1+4\sin^2\psi) \, d\psi=0.
$$
Moreover, it is easy to show that the following formula holds:
$$
(2+t)p'(t) = -\frac32 p(t)+ \frac32 \int_0^\pi\frac{(\sin^2\psi-2\cos^2\psi)^2(\sin^2\psi+4\sin^4\psi)}{(\sin^2\psi+t\cos^2\psi)^{5/2}} 
\sin\psi\, d\psi
$$
for every $t>0$.
From this equation one can immediately see that $p'(t)>0$ for $t\in(0,\varepsilon_0)$ for some $\varepsilon_0>0$, hence $p(t)>0$ for $t\in(0,\varepsilon_0)$. Moreover, if $p(t_0)=0$ for some $t_0>0$, then $p'(t_0)>0$.
We conclude that $p(t)=0$ only at $t=0$ and this proves the claim.
\end{example}

\begin{remark}[Quadratic profiles]\label{rem:quad}
We now show that for any quadratic profile $\Psi(x)=c_0+P_2(x)$, where $c_0$ is a constant and $P_2$ is a homogeneous polynomial of degree $2$,
with $\Psi>0$ and $\widehat{\Psi}\geq 0$ on $\mathbb{S}^2$ the minimiser of the corresponding energy has always a full-dimensional support, and hence loss of dimension does not occur in this case. This extends the result of \cite{CMMRSV2}, where the profile $\Psi_\alpha(x)=1+\alpha x_1^2 $ was considered for $\alpha\in (-1,1]$.

Since $c_0=c_0|x|^2$ for $x\in\mathbb{S}^2$ and by means of a diagonalisation procedure,
we can always reduce to the canonical form where $\Psi(x)=\alpha_1x_1^2+\alpha_2x_2^2+\alpha_3 x_3^2$, $x\in\mathbb{S}^2$.
The assumption $\Psi>0$ on $\mathbb{S}^2$ corresponds to $\alpha_i>0$ for $i=1,2,3$. 

To compute $\widehat{\Psi}$ it is convenient to rewrite $\Psi$ as a sum of homogeneous harmonic polynomials 
$$
\Psi(x) = \frac13\sum_{i=1}^3\alpha_i+\frac{1}3\sum_{i=1}^3 \alpha_i\bigg(2x_i^2-\sum_{j\neq i}x_j^2\bigg).
$$
Then, by \eqref{phievenbis} and \eqref{bcoeff} we have that for $\xi\in \mathbb{S}^2$
\begin{align}\label{Psi-quad}
\widehat{\Psi}(\xi) &= \frac13b_0\sum_{i=1}^3\alpha_i+\frac{1}3b_2\sum_{i=1}^3 \alpha_i\bigg(2\xi_i^2-\sum_{j\neq i}\xi_j^2\bigg)
\nonumber\\
&=\sqrt{\frac{2}{\pi}} \sum_{i=1}^3 \xi_i^2\bigg(-\alpha_i+\sum_{j\neq i}\alpha_j\bigg).
\end{align}
Note that the condition $\widehat{\Psi}\geq 0$ on $\mathbb{S}^2$ is guaranteed by requiring that $-\alpha_i+\sum_{j\neq i}\alpha_j\geq 0$ for every $i=1,2,3$. 

To analyse the case of a degenerate Fourier transform we assume, with no loss of generality, that the coefficient of $\xi_3^2$ in \eqref{Psi-quad} is zero, namely, that $\alpha_3=\alpha_1+\alpha_2$. In this case we have that 
$$
\Psi(x)=  \alpha_1x_1^2+\alpha_2x_2^2+(\alpha_1+\alpha_2) x_3^2
$$
and
$$
\widehat{\Psi}(\xi)=2\sqrt{\frac{2}{\pi}} \big(\alpha_2 \xi_1^2 + \alpha_1 \xi_2^2 \big).
$$
Under the assumption that $\alpha_1>0$ and $\alpha_2>0$ the profile $\Psi$ then satisfies the assumptions of Theorem~\textnormal{\ref{Mainthm}($\ref{degenerate}$)}. 

Let $I$ be the energy corresponding to the quadratic profile $\Psi$. We claim that the minimiser of $I$ is the normalised characteristic function of a non-degenerate ellipsoid. To prove the claim, we first note that, if the minimiser had a lower dimensional support, then the support would be in the plane orthogonal to the $x_3$ axis, since the coefficient of $x_3^2$  in the expression of $\Psi$ is larger than the coefficients of $x_1^2$ and $x_2^2$. Hence it remains to show that a semi-ellipsoid law supported in the $x_1x_2$-plane cannot be a minimiser.

We define a perturbed potential $\Psi_\varepsilon$ as in \eqref{pert-psi} and write the system \eqref{EL1-Fourier} for $\Psi_\varepsilon$. We know that the system has a solution $M_{\varepsilon}=\mathrm{diag}(A^{\varepsilon})$ with $A^{\varepsilon}=(A_1^{\varepsilon},A_2^{\varepsilon},A_3^{\varepsilon})\in\R^3$ such that
$A_i^{\varepsilon}=(a_i^{\varepsilon})^2>0$, $a_i^{\varepsilon}$ being the semi-axes of the minimising ellipsoid.

In particular, we have that for any $\varepsilon>0$
$$0=\int_{\mathbb{S}^2}\frac{(\omega_1^2-\omega_3^2)\widehat{\Psi_{\varepsilon}}(\omega)}{( A_1^{\varepsilon}\omega_1^2+A_2^{\varepsilon} \omega_2^2+A_3^{\varepsilon}\omega_3^2)^{3/2}}\,d\mathcal{H}^2(\omega)=
\int_{\mathbb{S}^2}\frac{(\omega_1^2-\omega_3^2)(\widehat{\Psi}(\omega)+\varepsilon)}{( A_1^{\varepsilon}\omega_1^2+A_2^{\varepsilon} \omega_2^2+A_3^{\varepsilon}\omega_3^2)^{3/2}}\,d\mathcal{H}^2(\omega).
$$
By contradiction, we assume that $(a_1^{\varepsilon},a_2^{\varepsilon},a_3^{\varepsilon})\to (a_1,a_2,0)$, as $\varepsilon \to 0^+$, with $a_1>0$, $a_2>0$. 
Setting $A_i:=a_i^2$, by the Dominated Convergence Theorem we have that
$$
\int_{\mathbb{S}^2}\frac{(\omega_1^2-\omega_3^2)\widehat{\Psi}(\omega)}{( A_1^{\varepsilon}\omega_1^2+A_2^{\varepsilon} \omega_2^2+A_3^{\varepsilon}\omega_3^2)^{3/2}}\,d\mathcal{H}^2(\omega)\to 
\int_{\mathbb{S}^2}\frac{(\omega_1^2-\omega_3^2)\widehat{\Psi}(\omega)}{( A_1\omega_1^2+A_2 \omega_2^2)^{3/2}}\,d\mathcal{H}^2(\omega)
$$
and
$$
\int_{\mathbb{S}^2}\frac{\omega_1^2}{( A_1^{\varepsilon}\omega_1^2+A_2^{\varepsilon} \omega_2^2+A_3^{\varepsilon}\omega_3^2)^{3/2}}\,d\mathcal{H}^2(\omega)\to 
\int_{\mathbb{S}^2}\frac{\omega_1^2}{( A_1\omega_1^2+A_2 \omega_2^2)^{3/2}}\,d\mathcal{H}^2(\omega).
$$
Since the remaining term
$$
- \int_{\mathbb{S}^2}\frac{\varepsilon \omega_3^2}{( A_1^{\varepsilon}\omega_1^2+A_2^{\varepsilon} \omega_2^2+A_3^{\varepsilon}\omega_3^2)^{3/2}}\,d\mathcal{H}^2(\omega)
$$
is negative, we conclude that
$$0=\lim_{\varepsilon\to 0^+}\int_{\mathbb{S}^2}\frac{(\omega_1^2-\omega_3^2)\widehat{\Psi_{\varepsilon}}(\omega)}{( A_1^{\varepsilon}\omega_1^2+A_2^{\varepsilon} \omega_2^2+A_3^{\varepsilon}\omega_3^2)^{3/2}}\,d\mathcal{H}^2(\omega)\leq
\int_{\mathbb{S}^2}\frac{(\omega_1^2-\omega_3^2)\widehat{\Psi}(\omega)}{( A_1\omega_1^2+A_2 \omega_2^2)^{3/2}}\,d\mathcal{H}^2(\omega)=\frac{p(t_\ast)}{a_2^3}
$$
where $t_\ast:=a_1^2/a_2^2$ and the function $p:(0,+\infty)\to\R$ is defined as
$$
p(t):= \int_{\mathbb{S}^2}\frac{(\omega_1^2-\omega_3^2)\widehat{\Psi}(\omega)}{(t \omega_1^2+ \omega_2^2)^{3/2}}\,d\mathcal{H}^2(\omega)
$$
for every $t>0$. By using spherical coordinates and the explicit expression of $\widehat{\Psi}$ we have that 
\begin{align*}
p(t) &= 2\sqrt{\frac{2}{\pi}}\int_{\mathbb{S}^2}\frac{(\omega_1^2-\omega_3^2)\big(\alpha_2 \omega_1^2 + \alpha_1 \omega_2^2 \big)}{(t \omega_1^2+ \omega_2^2)^{3/2}}\,d\mathcal{H}^2(\omega)\\
&=2\sqrt{\frac{2}{\pi}}\int_0^{2\pi}\int_0^\pi \frac{(\sin^2\psi\cos^2\theta-\cos^2\psi)\big(\alpha_2 \cos^2\theta + \alpha_1 \sin^2\theta \big)}{(t \cos^2\theta+ \sin^2\theta)^{3/2}}\,d\psi\, d\theta\\
&=\sqrt{2\pi}\int_0^{2\pi} \frac{\alpha_2 \cos^2\theta + \alpha_1 \sin^2\theta}{(t \cos^2\theta+ \sin^2\theta)^{3/2}}\,(\cos^2\theta-1)\,d\theta<0.
\end{align*}
Hence $p(t)<0$ for any $t>0$ and we obtain a contradiction.
\end{remark}

\bigskip

\noindent
\textbf{Acknowledgements.} MGM is member of GNAMPA--INdAM. 
JM and JV acknowledge support from the grants PID2020-112881GB-I00 and Severo Ochoa and Maria de Maeztu CEX2020-001084-M. 
MGM acknowledges support by MIUR--PRIN 2017. LR acknowledges support by  MIUR--PRIN 2017 project  201758MTR2 and by GNAMPA-INdAM through 2022 projects. LS acknowledges support by the EPSRC under the grants EP/V00204X/1 and EP/V008897/1. Part of this work was done during a visit of MGM, LR, and LS to the Universitat Aut\`onoma de Barcelona, whose kind hospitality is gratefully acknowledged.


\end{document}